\documentclass[11pt]{article}
\usepackage{amsthm,amsmath,amssymb}
\usepackage{graphicx}
\usepackage{latexsym}
\usepackage{epsfig}
\usepackage[]{times}
\theoremstyle{plain}
\newtheorem{theorem}{Theorem}

\newtheorem{corollary}{Corollary}
\newtheorem{lemma}{Lemma}
\theoremstyle{definition}
\newtheorem{definition}{Definition}

\theoremstyle{remark}

\def\la{\langle}
\def\ra{\rangle}
\def\be{\begin{equation}}
\def\ee{\end{equation}}
\def\bea{\begin{eqnarray}}
\def\eea{\end{eqnarray}}

\begin{document}
\title{Symmetric (not Complete Intersection) Numerical Semigroups Generated by 
Six Elements}
\author{Leonid G. Fel\\ \\
Department of Civil Engineering, Technion, Haifa 32000, Israel\\
{\sl e-mail: lfel@technion.ac.il}}
\date{}
\maketitle
\vspace{-.5cm}
%%%%%%%%%%%%%%%%%%%%%%%%%%%%%%%%%%%%%%%%%%%%%%%%%%%%%%%%%%%
\begin{abstract}
We consider symmetric (not complete intersection) numerical semigroups ${\sf S}_
6$, generated by a set of six positive integers $\{d_1,\ldots,d_6\}$, $\gcd(d_1,
\ldots,d_6)=1$, and derive inequalities for degrees of syzygies of such
semigroups and find the lower bound for their Frobenius numbers. We show that 
this bound may be strengthened if ${\sf S}_6$ satisfies the Watanabe lemma.
\end{abstract}
%%%%%%%%%%%%%%%%%%%%%%%%%%%%%%%%%%%%%%%%%%%%%%%%%%%%%%%%%%%%%%%%%%%%%%
\section{Symmetric numerical semigroups generated by six integers}\label{s1}
%%%%%%%%%%%%%%%%%%%%%%%%%%%%%%%%%%%%%%%%%%%%%%%%%
Let a numerical semigroup ${\sf S}_m=\la d_1,\ldots,d_m\ra$ be generated by a
set of positive integers $\{d_1,\ldots,d_m\}$, $d_1<\ldots<d_m$, such that 
$\gcd(d_1,\ldots,d_m)=1$, where $d_1$ and $m$ denote multiplicity and embedding 
dimension ({\em edim}) of ${\sf S}_m$. There exist $m-1$ polynomial identities 
\cite{fe017} for degrees of syzygies associated with semigroup ring $k[{\sf S}
_m]$. They are a source of various relations for semigroups of different nature.
In the case of complete intersection (CI) semigroups such relation for degrees 
of the 1st syzygy was found in \cite{fe017}, Corollary 1. The next nontrivial 
case exhibits a symmetric (not CI) semigroup generated by $m\ge 4$ integers. In 
\cite{fe016} and \cite{fe018}, such semigroups with $m=4$ and $m=5$ were studied
and the lower bound for the Frobenius numbers $F({\sf S}_4)$ and $F({\sf S}_5)$ 
were found. In the present paper we deal with more difficult case of symmetric 
(not CI) semigroups ${\sf S}_6$.

Consider a symmetric numerical semigroup ${\sf S}_6$, which is not CI and 
generated by six positive integers. Its Hilbert series $H\left({\sf S}_6;t
\right)$ with independent Betti's numbers $\beta_1$, $\beta_2$ reads:
\bea
&&H\left({\sf S}_6;t\right)=\frac{Q_6(t)}{\prod_{i=1}^6\left(1-t^{d_i}\right)},
\label{j1}\\
&&Q_6(t)=1-\sum_{j=1}^{\beta_1}t^{x_j}+\sum_{j=1}^{\beta_2}t^{y_j}-\sum_{j=1}^
{\beta_2}t^{g-y_j}+\sum_{j=1}^{\beta_1}t^{g-x_j}-t^{g},\nonumber\\
&&\hspace{3.5cm}x_j,y_j,g\in{\mathbb Z}_{>},\qquad 2d_1\le x_j,y_j<g.\nonumber
\eea
The Frobenius number $F({\sf S}_6)$ of numerical semigroup ${\sf S}_6$ is 
related to the largest degree $g$ as follows:
\bea
F({\sf S}_6)=g-\sigma_1,\qquad \sigma_1=\sum_{j=1}^6d_j.\nonumber
\eea
There are two constraints more, $\beta_1>5$ and $d_1>6$. The inequality $\beta_1
>5$ holds since ${\sf S}_6$ is not CI, and the condition $d_1>6$ is necessary 
since a semigroup $\la m,d_2,\ldots,d_m\ra$ is never symmetric \cite{fe011}.
%%%%%%%%%%%%%%%%%%%%%%%%%%%%%%%%%%%%%%%%%%%%%%%%%%%%%%%%%%%%%%%%%%%%%%
\section{Polynomial identities for degrees of syzygies}\label{s2}
%%%%%%%%%%%%%%%%%%%%%%%%%%%%%%%%%%%%%%%%%%%%%%%%%
Polynomial identities for degrees of syzygies for numerical semigroups were
derived in \cite{fe017}, Thm 1. In the case of symmetric (not CI) semigroup
${\sf S}_6$, they read:
\bea
&&\sum_{j=1}^{\beta_1}x_j^r-\sum_{j=1}^{\beta_2}y_j^r+\sum_{j=1}^{\beta_2}
(g-y_j)^r-\sum_{j=1}^{\beta_1}(g-x_j)^r+g^r=0,\quad r\le 4,\label{j2}\\
&&\sum_{j=1}^{\beta_1}x_j^5-\sum_{j=1}^{\beta_2}y_j^5+\sum_{j=1}^{\beta_2}
(g-y_j)^5-\sum_{j=1}^{\beta_1}(g-x_j)^5+g^5=120\pi_6,\quad\pi_6=\prod_{j=1}^6d
_j.\nonumber
\eea
Only three of five identities in (\ref{j2}) are not trivial, these are for 
$r=1,3,5$:
\bea
&&{\cal B}_6g+\sum_{j=1}^{\beta_1}x_j=\sum_{j=1}^{\beta_2}y_j,\hspace{2.5cm}
{\cal B}_6=\frac{\beta_2-\beta_1+1}{2},\label{j3}\\
&&{\cal B}_6g^3+\sum_{j=1}^{\beta_1}x_j^2\left(3g-2x_j\right)=
\sum_{j=1}^{\beta_2}y_j^2\left(3g-2y_j\right),\label{j4}\\
&&{\cal B}_6g^5+\sum_{j=1}^{\beta_1}x_j^3\left(10g^2-15gx_j+6x_j^2\right)-
360\pi_6=\sum_{j=1}^{\beta_2}y_j^3\left(10g^2-15gy_j+6y_j^2\right).\label{j5}
\eea
where ${\cal B}_6$ is defined according to the expression for an arbitrary 
symmetric semigroup ${\sf S}_m$ in \cite{fe017}, Formulas (5.7, 5.9). The sign 
of ${\cal B}_6$ is strongly related to the famous Stanley Conjecture 4b 
\cite{sta89} on the unimodal sequence of Betti's numbers in the 1-dim local 
Gorenstein rings $k[{\sf S}_m]$. We give its simple proof in the case $edim=6$.
%%%%%%%%%%%%%%%%%%%%%%%%%%%%%%%%%%%%%%%%%%%%%%%%%%%%%%%%
\begin{lemma}\label{le1}
%%%%%%%%%%%%%%%%%%%%%%%%%%%%%%%%%%%%%%%%%%%%%%%%%%%%%%%%
Let a symmetric (not CI) semigroup ${\sf S}_6$ be given with the Hilbert series
$H\left({\sf S}_6;z\right)$ in accordance with (\ref{j1}). Then 
\bea
\beta_2\ge\beta_1+1.\label{j6}
\eea
\end{lemma}
%%%%%%%%%%%%%%%%%%%%%%%%%%%%%%%%%%%%%%%%%%%
\begin{proof}
According to the identity (\ref{j3}) and constraints on degrees $x_j$ of the 
1st syzygies (\ref{j1}) we have,
\bea
\sum_{j=1}^{\beta_2}y_j<{\cal B}_6g+\beta_1g=\frac{\beta_2+\beta_1+1}{2}g.
\label{j7}
\eea
On the other hand, there holds another constraint on degrees $y_j$ of the 2nd 
syzygies,
\bea
\sum_{j=1}^{\beta_2}y_j<\beta_2g.\label{j8}
\eea
Inequality (\ref{j8}) holds always, while inequality (\ref{j7}) is not valid
for every set $\{x_1,\ldots,x_{\beta_1}\}$, but only when (\ref{j3}) holds. In 
order to make the both inequalities consistent, we have to find a relation 
between $\beta_1$ and $\beta_2$ where both inequalities (\ref{j7}) and 
(\ref{j8}) are satisfied, even if (\ref{j7}) is stronger than (\ref{j8}). To 
provide these inequalities to be correct, it is enough to require $(\beta_2+
\beta_1+1)/2\le\beta_2$, that leads to (\ref{j6}).
\end{proof}
%%%%%%%%%%%%%%%%%%%%%%%%%%%%%%%%%%%%%%%%%%%
Another constraint for Betti's numbers $\beta_j$ follows from the general 
inequality for the sum of $\beta_j$ in the case of non-symmetric semigroups 
\cite{fe011}, Formula (1.9),
\bea
\sum_{j=0}^{m-1}\beta_j\le d_12^{m-1}-2(m-1),\qquad\beta_0=1.\label{j9}
\eea
Applying the duality relation for Betti's numbers, $\beta_j=\beta_{m-j-1}$, 
$\beta_{m-1}=1$, in symmetric semigroups ${\sf S}_6$ to inequality (\ref{j9}) 
and combining it with Lemma \ref{le1}, we obtain
\bea
\beta_1<2(4d_1-1).\label{j10}
\eea

To study polynomial identities (\ref{j3},\ref{j4},\ref{j5}) and their 
consequences, start with observation, which follows by numerical calculations 
for two real functions $R_1(z)$, $R_2(z)$ and is presented in Figure \ref{fig1},
\bea
&&R_1(z)\ge A_*R_2(z),\quad 0\le z\le 1,\qquad\mbox{where}\label{j11}\\
&&R_1(z)=z^2\sqrt{10-15z+6z^2},\quad R_2(z)=z^2(3-2z),\quad A_*=0.9682.\nonumber
\eea
The constant $A_*$ is chosen by requirement of the existence of such a 
coordinate $z_*\in[0,1]$ providing
%%%%%%%%%%%%%%%%%%%%%%%%%%%%%%%%%%%%%%%%%%%5
\begin{figure}[h!]\begin{center}\begin{tabular}{cc}
\psfig{figure=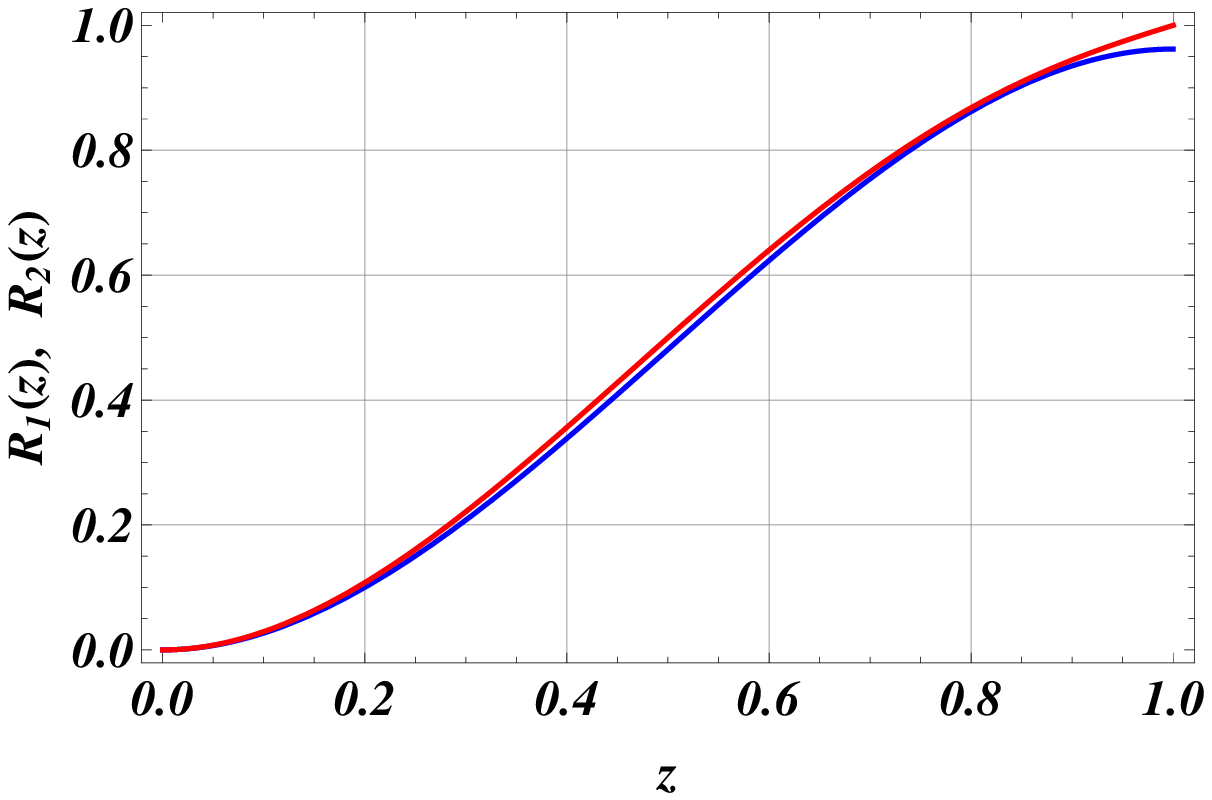,height=4.7cm} &
\psfig{figure=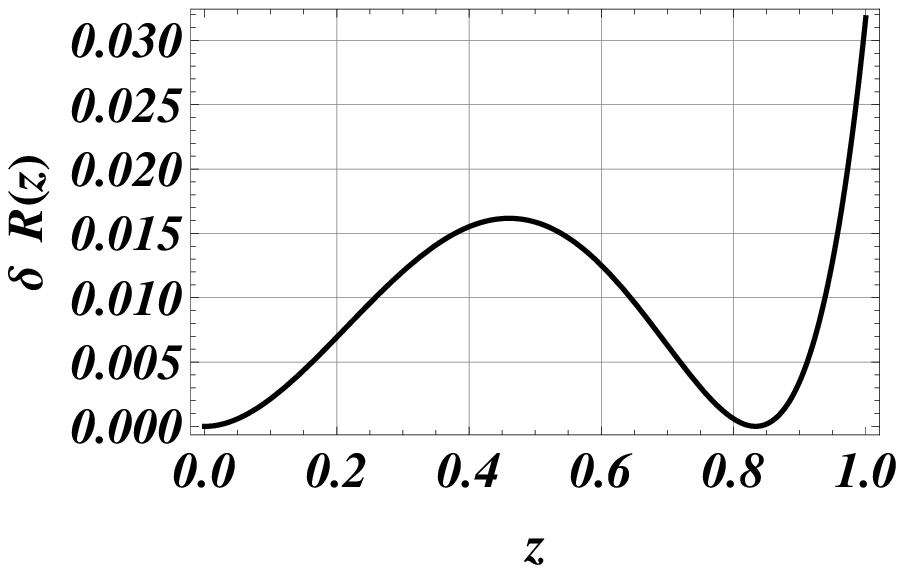,height=4.7cm}\\
(a) & (b)
\end{tabular}\end{center}
\vspace{-.5cm}
\caption{Plot of the functions (a) $R_1(z)$ {\em in red color}, $A_*R_2(z)$ 
{\em in blue color} and a discrepancy (b) $\delta R(z)=R_1(z)-A_*R_2(z)$ in the 
range $z\in[0,1]$.}\label{fig1}
\end{figure}

\noindent
two equalities,
\bea
R_1(z_*)=A_*R_2(z_*),\qquad R_1^{\prime}(z_*)=A_*R_2^{\prime}(z_*),\quad z_*
\simeq 0.8333,\quad\mbox{where}\quad R_j^{\prime}(z_*)=\frac{dR_j(z)}{dz}_{|z=
z_*}.\nonumber
\eea
Substituting $z=y_j/g$, $0\!<\!z\!<\!1$, into inequality (\ref{j11}) and making 
summation over $1\le j\le\beta_2$, we get
\bea
A_*\sum_{j=1}^{\beta_2}y_j^2(3g-2y_j)<\sum_{j=1}^{\beta_2}y_j^2\sqrt{10g^2-
15gy_j+6y_j^2}.\label{j12}
\eea
Applying the Cauchy-Schwarz inequality $\left(\sum_{j=1}^Na_jb_j\right)^2\le
\left(\sum_{j=1}^Na_j^2\right)\left(\sum_{j=1}^Nb_j^2\right)$ to the right-hand 
side of inequality (\ref{j12}), we obtain
\bea
\left(\sum_{j=1}^{\beta_2}y_j^{3/2}\sqrt{10g^2-15gy_j+6y_j^2}\;y_j^{1/2}\right)
^2\le\sum_{j=1}^{\beta_2}y_j^3(10g^2-15gy_j+6y_j^2)\sum_{j=1}^{\beta_2}y_j.
\label{j13}
\eea
Combining (\ref{j12}) and (\ref{j13}), we arrive at inequality
\bea
A_*^2\left(\sum_{j=1}^{\beta_2}y_j^2(3g-2y_j)\right)^2<\sum_{j=1}^{\beta_2}y_j
^3(10g^2-15gy_j+6y_j^2)\sum_{j=1}^{\beta_2}y_j.\label{j14}
\eea
Denote by $X_k$ the k-th power symmetric polynomial $X_k(x_1,\ldots,x_{\beta_1})
=\sum_{j=1}^{\beta_1}x_j^k$, $x_j<g$, and substitute identities 
(\ref{j3},\ref{j4},\ref{j5}) into inequality (\ref{j14}),
\bea
A_*^2\left({\cal B}_6g^3+3gX_2-2X_3\right)^2<\left({\cal B}_6g^5-360\pi_6+10g^2
X_3-15gX_4+6X_5\right)\left({\cal B}_6g+X_1\right).\label{j15}
\eea

On the other hand, similarly to inequalities (\ref{j12},\ref{j13},\ref{j14}), 
let us establish another set of inequalities for $X_k$ by replacing $y_j\to 
x_j$. We write the last of them, which is similar to (\ref{j14}),
\bea
A_*^2\left(\sum_{j=1}^{\beta_1}x_j^2(3g-2x_j)\right)^2<\sum_{j=1}^{\beta_1}
x_j^3(10g^2-15gx_j+6x_j^2)\sum_{j=1}^{\beta_1}x_j,\label{j16}
\eea
and present (\ref{j16}) in terms of $X_k$,
\bea
A_*^2\left(3gX_2-2X_3\right)^2<\left(10g^2X_3-15gX_4+6X_5\right)X_1.\label{j17}
\eea
Represent the both inequalities (\ref{j15}) and (\ref{j17}) as follows:
\bea
360\pi_6-{\cal B}_6g^5+A_*^2\frac{\left({\cal B}_6g^3+3gX_2-2X_3\right)^2}{
{\cal B}_6g+X_1}<10g^2X_3-15gX_4+6X_5,\label{j18}\\
A_*^2\frac{\left(3gX_2-2X_3\right)^2}{X_1}<10g^2X_3-15gX_4+6X_5.\label{j19}
\eea
Inequality (\ref{j19}) holds always, while inequality (\ref{j18}) is not valid 
for every set $\{x_1,\ldots,x_{\beta_1},g\}$. In order to make the both 
inequalities consistent, we have to find a range of $g$ where both inequalities 
(\ref{j18}) and (\ref{j19}) are satisfied. To provide these inequalities to be 
correct, it is enough to require that inequality (\ref{j19}) implies inequality 
(\ref{j18}), i.e.,
\bea
\frac{360\pi_6-{\cal B}_6g^5}{A_*^2}+\frac{\left({\cal B}_6g^3+3gX_2-2X_3\right)
^2}{{\cal B}_6g+X_1}<\frac{\left(3gX_2-2X_3\right)^2}{X_1}.\label{j20}
\eea
Simplifying the above expressions, we present the last inequality (\ref{j20}) 
as follows:
\bea
CX_1(X_1+{\cal B}_6g)<\left(3gX_2-2X_3-g^2X_1\right)^2,\qquad C=\frac{360\pi_6-
\alpha{\cal B}_6g^5}{A_*^2{\cal B}_6g},\label{j21}
\eea
where $\alpha=1-A_*^2\simeq 0.06259$ and ${\cal B}_6\ge 1$ due to Lemma 
\ref{le1}. An inequality (\ref{j21}) holds always if its left-hand side is 
negative, i.e., $C<0$, that results in the following constraint,
\bea
g>q_6,\qquad q_6=\sqrt[5]{\frac{360}{\alpha\;{\cal B}_6}}\sqrt[5]{\pi_6},\quad
\mbox{where}\qquad\sqrt[5]{\frac{360}{\alpha}}\simeq 5.649.\label{j22}
\eea
The lower bound $q_6$ in (\ref{j22}) provides a sufficient condition to satisfy 
the inequality (\ref{j21}). In fact, a necessary condition has to produce 
another bound $g_6<q_6$.
%%%%%%%%%%%%%%%%%%%%%%%%%%%%%%%%%%%%%%%%%%%%%%%%%
\section{The lower bound for the Frobenius numbers of semigroups ${\sf S}_6$}
\label{s3}
%%%%%%%%%%%%%%%%%%%%%%%%%%%%%%%%%%%%%%%%%%%%%%%%%
An actual lower bound of $g$ precedes that, given in (\ref{j22}), since the
inequality (\ref{j21}) may be satisfied for a sufficiently small $C>0$. To 
find it, we introduce another kind of symmetric polynomials ${\cal X}_k$: 
\bea
&&\hspace{4cm}{\cal X}_k=\sum_{i_1<i_2<\ldots<i_k}^{\beta_1}x_{i_1}x_{i_2}
\ldots x_{i_k},\nonumber\\
&&{\cal X}_0=1,\quad {\cal X}_1=\sum_{i=1}^{\beta_1}x_i,\quad {\cal X}_2=\sum_{
i<j}^{\beta_1}x_ix_j,\quad {\cal X}_3=\sum_{i<j<r}^{\beta_1}x_ix_jx_r,\quad
\ldots,\quad {\cal X}_{\beta_1}=\prod_{i=1}^{\beta_1}x_i,\nonumber
\eea
which are related to polynomials $X_k$ by the Newton recursion identities,
\bea
m{\cal X}_m=\sum_{k=1}^m(-1)^{k-1}X_k{\cal X}_{m-k},\quad\mbox{i.e.,}
\hspace{2cm}\nonumber\\
X_1={\cal X}_1,\quad X_2={\cal X}_1^2-2{\cal X}_2,\quad X_3={\cal X}_1^3-
3{\cal X}_2{\cal X}_1+3{\cal X}_3,\quad\ldots\;\;.\label{j23}
\eea
Recall the Newton-Maclaurin inequalities \cite{har59} for polynomials 
${\cal X}_k$,
\bea
\frac{{\cal X}_1}{\beta_1}\ge\left(\frac{{\cal X}_2}{{\beta_1\choose 2}}\right)
^{\frac1{2}}\ge\left(\frac{{\cal X}_3}{{\beta_1\choose 3}}\right)^{\frac1{3}}
\ge\ldots\ge\sqrt[\beta_1]{{\cal X}_{\beta_1}}.\label{j24}
\eea
Consider the master inequality (\ref{j21}) in the following form
\bea
CX_1(X_1+{\cal B}_6g)<9g^2X_2^2+4X_3^2+g^4X_1^2+4g^2X_1X_3-12gX_2X_3-6g^3X_1X_2,
\label{j25}
\eea
and substitute Newton's identities (\ref{j23}) into (\ref{j25}),
\bea
{\cal X}_1P({\cal X}_1,{\cal X}_2,{\cal X}_3)<{\cal X}_1Q_1({\cal X}_1)+
{\cal X}_2Q_2({\cal X}_1,{\cal X}_2)+{\cal X}_3Q_3({\cal X}_1,{\cal X}_2,
{\cal X}_3),\label{j26}
\eea
where
\bea
P({\cal X}_1,{\cal X}_2,{\cal X}_3)&=&4g^2{\cal X}_1{\cal X}_2+2{\cal X}_1^3
{\cal X}_2+6{\cal X}_2{\cal X}_3+6g{\cal X}_2^2+3g{\cal X}_1{\cal X}_3,
\nonumber\\
Q_1({\cal X}_1)&=&\frac1{3}{\cal X}_1^5-g{\cal X}_1^4+\frac{13}{12}g^2{\cal X}_1
^3-\frac1{2}g^3{\cal X}_1^2+\frac{g^4-C}{12}{\cal X}_1-\frac{C}{12}{\cal B}_6g,
\nonumber\\
Q_2({\cal X}_1,{\cal X}_2)&=&3g^2{\cal X}_2+3{\cal X}_2{\cal X}_1^2+
5g{\cal X}_1^3+g^3{\cal X}_1,\nonumber\\
Q_3({\cal X}_1,{\cal X}_2,{\cal X}_3)&=&3{\cal X}_3+2{\cal X}_1^3+
g^2{\cal X}_1+6g{\cal X}_2.\nonumber
\eea
Applying inequalities (\ref{j24}) to $Q_2({\cal X}_1,{\cal X}_2)$ and $Q_3(
{\cal X}_1,{\cal X}_2,{\cal X}_3)$, we obtain
\bea
Q_2({\cal X}_1,{\cal X}_2)\!\!&\!<\!&\!\!{\cal X}_1Q_{21}({\cal X}_1),\quad
Q_{21}({\cal X}_1)=3\frac{{\cal X}_1}{\beta_1^2}{\beta_1\choose 2}\left(g^2+   
{\cal X}_1^2\right)+5g{\cal X}_1^2+g^3,\nonumber\\
Q_3({\cal X}_1,{\cal X}_2,{\cal X}_3)\!&\!<\!&\!{\cal X}_1Q_{31}({\cal X}_1),
\quad Q_{31}({\cal X}_1)=3\frac{{\cal X}_1^2}{\beta_1^3}{\beta_1\choose 3}+
2{\cal X}_1^2+g^2+6g\frac{{\cal X}_1}{\beta_1^2}{\beta_1\choose 2}.\label{j27}
\eea
Substituting inequalities (\ref{j27}) into (\ref{j26}) and applying again
(\ref{j24}), we obtain
\bea
P({\cal X}_1,{\cal X}_2,{\cal X}_3)<Q_1({\cal X}_1)+\frac{{\cal X}_1^2}
{\beta_1^2}{\beta_1\choose 2}Q_{21}({\cal X}_1)+\frac{{\cal X}_1^3}{\beta_1^3} 
{\beta_1\choose 3}Q_{31}({\cal X}_1).\label{j28}
\eea
Represent the right-hand side of inequality (\ref{j28}) as a polynomial 
$E({\cal X}_1)$ of the 5th order in ${\cal X}_1$,
\bea
&&E({\cal X}_1)=\sum_{k=0}^5E_kg^{5-k}{\cal X}_1^k,\qquad\mbox{where}
\label{j29}\\
&&E_0=-\frac{{\cal B}_6Cg^{-4}}{12},\qquad E_1=\frac{1-Cg^{-4}}{12},\qquad
E_2=\frac1{\beta_1^2}{\beta_1\choose 2}-\frac1{2}=-\frac1{2\beta_1},\nonumber\\
&&E_3=\frac{3}{\beta_1^4}{\beta_1\choose 2}^2+\frac1{\beta_1^3}{\beta_1
\choose 3}+\frac{13}{12},\qquad E_4=\frac{5}{\beta_1^2}{\beta_1\choose 2}+
\frac{6}{\beta_1^5}{\beta_1\choose 2}{\beta_1\choose 3}-1,\nonumber\\
&&E_5=\frac{3}{\beta_1^4}{\beta_1\choose 2}^2+\frac{3}{\beta_1^6}{\beta_1
\choose 3}^2+\frac{2}{\beta_1^3}{\beta_1\choose 3}+\frac1{3}.\nonumber
\eea
Thus, the master inequality (\ref{j21}) reads:
\bea
P({\cal X}_1,{\cal X}_2,{\cal X}_3)<E({\cal X}_1).\label{j30}
\eea
On the other hand, applying (\ref{j24}) to the polynomial $P({\cal X}_1,
{\cal X}_2,{\cal X}_3)$, we have another inequality,
\bea
P({\cal X}_1,{\cal X}_2,{\cal X}_3)<J({\cal X}_1),\qquad 
J({\cal X}_1)=\sum_{k=3}^5J_kg^{5-k}{\cal X}_1^k,\label{j31}\\
J_5=\frac{2}{\beta_1^2}{\beta_1\choose 2}\left[1+\frac{3}{\beta_1^3}
{\beta_1\choose 3}\right],\qquad J_4=\frac{6}{\beta_1^4}{\beta_1\choose 2}^2+
\frac{3}{\beta_1^3}{\beta_1\choose 3},\qquad J_3=\frac{4}{\beta_1^2}{\beta_1
\choose 2}.\nonumber
\eea
Inequality (\ref{j31}) holds always, while inequality (\ref{j30}) is not valid 
for every set $\{x_1,\ldots,x_{\beta_1},g\}$. In order to make both inequalities
consistent, we have to find a range for $g$ where both inequalities (\ref{j30})
and (\ref{j31}) are satisfied. To provide both inequalities to be correct, it
is enough to require that (\ref{j31}) implies (\ref{j30}),
\bea
E({\cal X}_1)>J({\cal X}_1),\qquad\mbox{or}\hspace{4cm}\nonumber\\
(E_5-J_5){\cal X}_1^5+(E_4-J_4)g{\cal X}_1^4+(E_3-J_3)g^2{\cal X}_1^3+E_2g^3
{\cal X}_1^2+E_1g^4{\cal X}_1+E_0g^5>0,\label{j32}
\eea
\bea
&&E_5-J_5=\frac{3}{\beta_1^4}{\beta_1\choose 2}^2+\frac{3}{\beta_1^6}
{\beta_1\choose 3}^2+\frac{2}{\beta_1^3}{\beta_1\choose 3}+\frac1{3}-
\frac{2}{\beta_1^2}{\beta_1\choose 2}\left[1+\frac{3}{\beta_1^3}
{\beta_1\choose 3}\right]=\frac1{3\beta_1^4},\nonumber\\
&&E_4-J_4=\frac{5}{\beta_1^2}{\beta_1\choose 2}+\frac{6}{\beta_1^5}{\beta_1
\choose 2}{\beta_1\choose 3}-1-\frac{6}{\beta_1^4}{\beta_1\choose 2}^2-
\frac{3}{\beta_1^3}{\beta_1\choose 3}=-\frac1{\beta_1^3},\nonumber\\
&&E_3-J_3=\frac{3}{\beta_1^4}{\beta_1\choose 2}^2+\frac1{\beta_1^3}{\beta_1
\choose 3}+\frac{13}{12}-\frac{4}{\beta_1^2}{\beta_1\choose 2}=
\frac{13}{12\beta_1^2}.\label{j33}
\eea
Substituting expressions $E_k-J_k$, $k=3,4,5$ from (\ref{j33}) and 
$E_0,E_1,E_2$ from (\ref{j29}) into (\ref{j32}), we obtain
\bea
\frac{C}{g^4}<G(b,u),\qquad G(b,u)=\frac{u}{u+b}(1-u)^2(1-2u)^2,\quad u=
\frac{{\cal X}_1}{\beta_1g},\quad b=\frac{{\cal B}_6}{\beta_1}.\label{j34}
\eea
%%%%%%%%%%%%%%%%%%%%%%%%%%%%%%%%%%%%%%%%%%%5
\begin{figure}[h!]\begin{center}
\psfig{figure=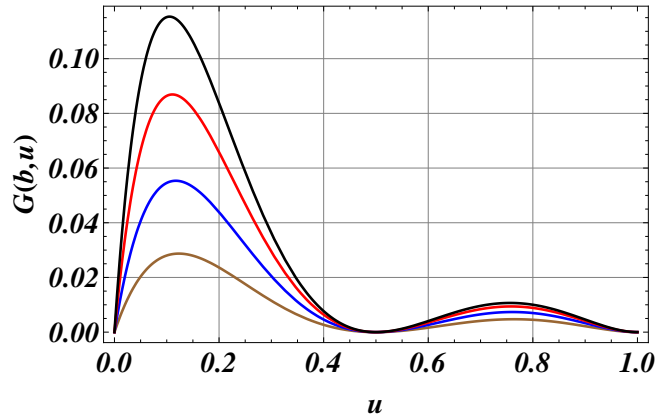,height=5.5cm}\\
\end{center}
\vspace{-.5cm}
\caption{Plot of the functions $G(b,u)$ with different $b$: ({\em in brown}) 
$b=1.75$, $u_m=0.125$; ({\em in blue}) $b=0.85$, $u_m=0.117$; ({\em in red}) 
$b=0.5$, $u_m=0.112$; ({\em in black}) $b=0.35$, $u_m=0.107$.}\label{fig2}
\end{figure}
%%%%%%%%%%%%%%%%%%%%%%%%%%%%%%%%%%%%%%%%%%%%%%%%%  
The function $G(b,u)$ is continuous (see Figure \ref{fig2}) and attains its 
global maximal value $G(b,u_m)$ at $u_m(b)\in(0,1/2)$, where $u_m=u_m(b)$ is a 
smaller positive root of cubic equation,
\bea
8u_m^3+2(5b-3)u_m^2-9bu_m+b=0,\nonumber
\eea
with asymptotic behavior of $u_m(b)$ and $G(b,u_m)$ (see Figure \ref{fig3}),
\bea
&&u_m(b)\stackrel{b\to 0}{\longrightarrow}\sqrt{\frac{b}{6}},\quad u_m(b)
\stackrel{b\to\infty}{\longrightarrow}v_1-\frac{v_2}{b},\quad v_1\!=\!\frac{9-
\sqrt{41}}{20}\simeq 0.1298,\;\;v_2\!=\!\frac{7\sqrt{41}-3}{500\sqrt{41}}
\simeq 0.013,\nonumber\\
&&G(b,u_m)\stackrel{b\to 0}{\longrightarrow}1,\quad G(b,u_m)\stackrel{b\to
\infty}{\longrightarrow}\frac{w}{b},\quad w=\frac{411+41\sqrt{41}}{12500}
\simeq 0.05388.\label{j35}
\eea
%%%%%%%%%%%%%%%%%%%%%%%%%%%%%%%%%%%%%%%%%%%%%%%%%%%%%%%%
\begin{theorem}\label{th1}
%%%%%%%%%%%%%%%%%%%%%%%%%%%%%%%%%%%%%%%%%%%%%%%%%%%%%%%%
Let a symmetric (not CI) semigroup ${\sf S}_6$ be given with its Hilbert 
series $H\left({\sf S}_6;z\right)$ in accordance with (\ref{j1}). Then the 
following inequality holds:
\bea
g>g_6,\qquad g_6=\lambda_6\sqrt[5]{\pi_6},\qquad\lambda_6=\sqrt[5]{\frac{360}
{{\cal B}_6K(b,A_*)}},\qquad K(b,A_*)=\alpha+A_*^2G(b,u_m).\label{j36}
\eea
\end{theorem}
%%%%%%%%%%%%%%%%%%%%%%%%%%%%%%%%%%%%%%%%%%%%%%%%%%%%%%%%
\begin{proof}
Substitute into (\ref{j34}) the expression for $C$, given in (\ref{j21}), and 
arrive at inequality
\bea
\frac{360\pi_6-\alpha{\cal B}_6g^5}{A_*^2{\cal B}_6g^5}<G(b,u_m),\nonumber
\eea
which gives rise to the lower bound $g_6$ in (\ref{j36}).
\end{proof}
%%%%%%%%%%%%%%%%%%%%%%%%%%%%%%%%%%%%%%%%%%%%%%%%%%%%%%%% 
\begin{figure}[h!]\begin{center}\begin{tabular}{cc}
\psfig{figure=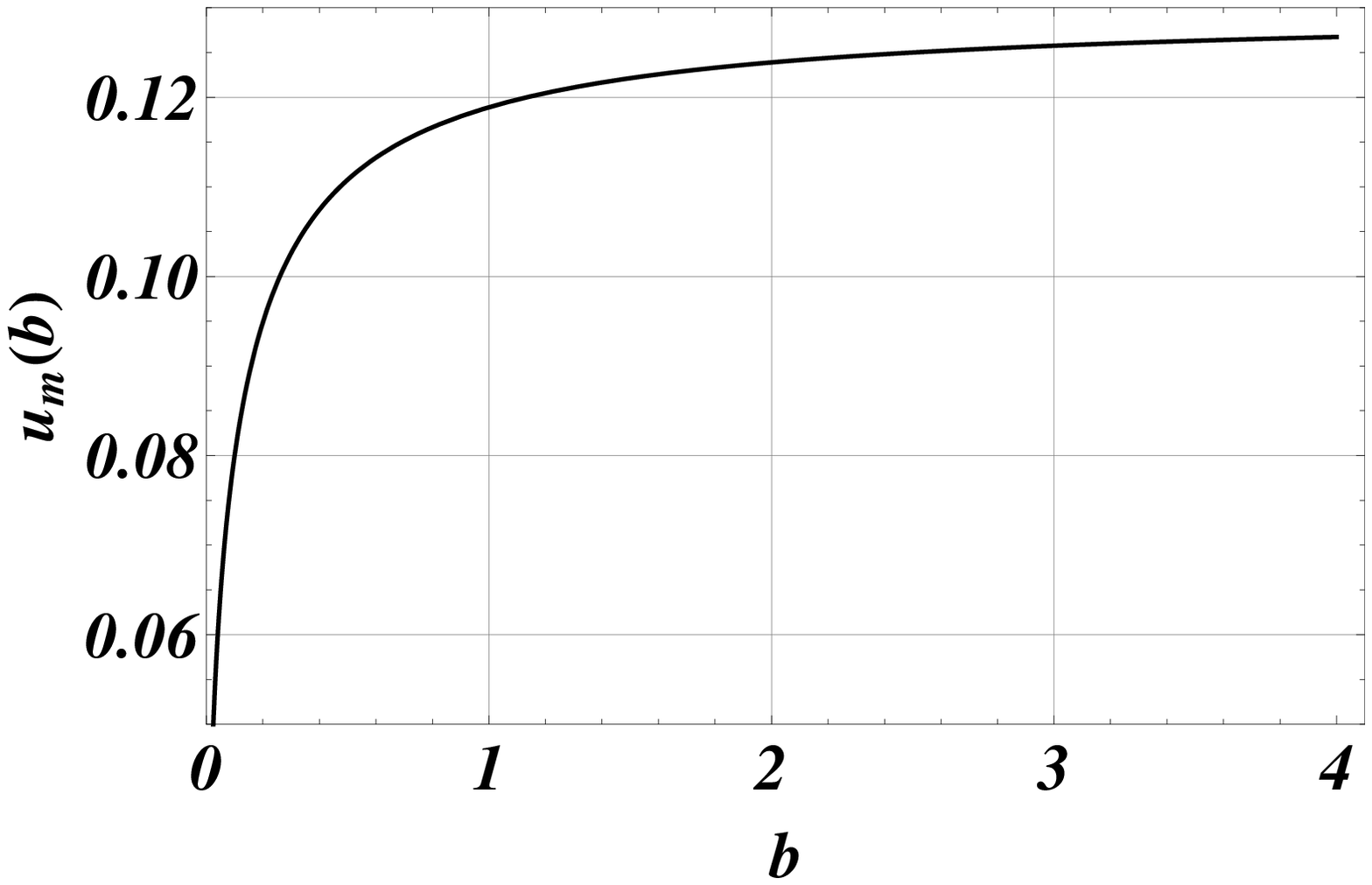,height=4.9cm} &
\psfig{figure=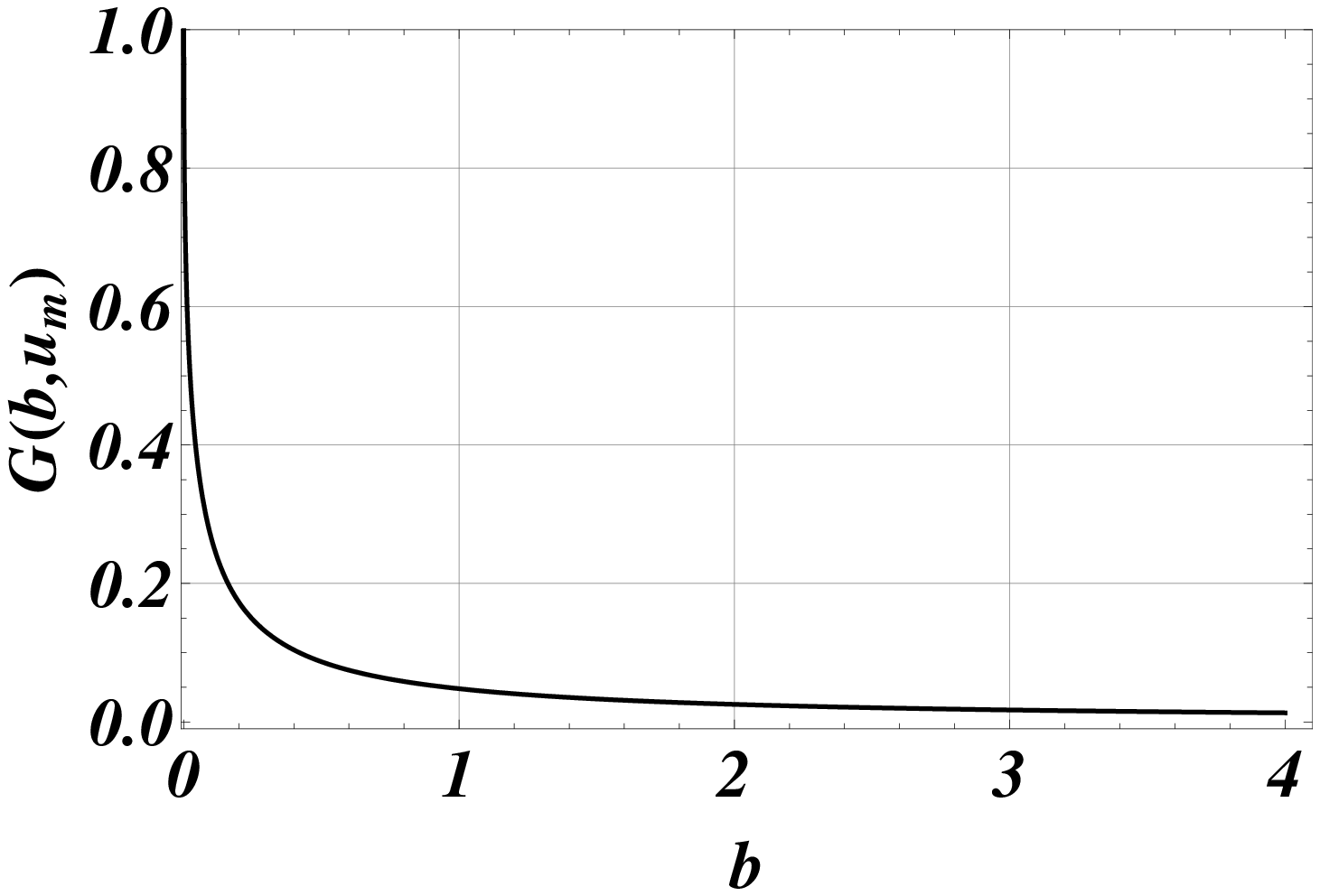,height=4.9cm}\\
(a) & (b)
\end{tabular}\end{center}
\vspace{-.5cm}
\caption{Plot of the functions (a) $u_m(b)$ and (b) $G(b,u_m)$ in a wide
range of $b$.}\label{fig3}
\end{figure}  
%%%%%%%%%%%%%%%%%%%%%%%%%%%%%%%%%%%%%%%%%%%%%%%%%%%%%%%
Formula for $\lambda_6$ in (\ref{j36}) shows a strong dependence on ${\cal B}_
6$, even the last is implicitly included into $G(b,u_m)$ by a slowly growing 
function $u_m(b)$ when $b>1$. Such dependence $\lambda_6({\cal B}_6)$ may lead 
to a very small values of $\lambda_6$ if ${\cal B}_6$ is not bounded from above,
but $b$ is fixed, and results in an asymptotic decrease of the bound, $g_6
\stackrel{{\cal B}_6\to\infty}{\longrightarrow}0$. The last limit poses a 
question: does formula (\ref{j36}) for $g_6$ contradict the known lower bound 
\cite{kil00} for the Frobenius number in the 6-generated numerical semigroups of
the arbitrary nature, i.e., not assuming their symmetricity. If the answer is 
affirmative then it arises another question: what should be required in order 
to avoid such contradiction. We address both questions in the next section in 
a slightly different form: are there any constraints on Betti's numbers.
%%%%%%%%%%%%%%%%%%%%%%%%%%%%%%%%%%%%%%%%%%%
\section{Are there any constraints on Betti's numbers of symmetric (not CI)
semigroups ${\sf S}_6$}\label{s4}
%%%%%%%%%%%%%%%%%%%%%%%%%%%%%%%%%%%%%%%%%%%%%%%%%
Denote by $\widetilde{g_6}$ and $\overline{g_6}$ the lower bounds of the 
largest degree of syzygies for non-symmetric \cite{kil00} and symmetric CI 
\cite{fe017} semigroups generated by six integers, respectively. Compare $g_6$ 
with $\widetilde{g_6}$ and $\overline{g_6}$ and require that the following 
double inequality hold:
\bea
\widetilde{g_6}<g_6<\overline{g_6},\qquad\widetilde{g_6}=\sqrt[5]{120}
\sqrt[5]{\pi_6},\quad\overline{g_6}=5\sqrt[5]{\pi_6}.\label{j37}
\eea
Substituting the expression for $g_6$ from (\ref{j36}) into (\ref{j37}), we 
obtain
\bea
\frac{72}{625}\frac1{K(b,A_*)}<{\cal B}_6<\frac{3}{K(b,A_*)},\qquad\frac{72}
{625}=0.1152,\quad K(b,A_*)\stackrel{b\to 0}{\longrightarrow}1,\quad K(b,A_*)
\stackrel{b\to\infty}{\longrightarrow}\alpha,\label{j38}
\eea
where the two limits follow by (\ref{j35},\ref{j36}). The double inequality
(\ref{j38}) determines the upper and lower bounds for varying ${\cal B}_6$ in 
the plane $(b,{\cal B}_6)$ as monotonic functions (see Figure \ref{fig4}a) 
with asymptotic behavior,
\bea
Upp.\;bound:\quad\!{\cal B}_6\stackrel{b\to 0}{\longrightarrow}3,\;\;{\cal B}_6
\stackrel{b\to\infty}{\longrightarrow}47.92;\quad Low.\;Bound:\quad\!{\cal B}_6
\stackrel{b\to 0}{\longrightarrow}0.1152,\;\;{\cal B}_6\stackrel{b\to\infty}
{\longrightarrow}1.84.\quad\label{j39}
\eea
According to Lemma \ref{le1}, the lower bound in (\ref{j38},\ref{j39}) may be 
chosen as ${\cal B}_6=1$. 

Find the constraints on Betti's numbers. For this purpose, the inequality 
(\ref{j38}) has to be replaced by
\bea
1<\beta_2-\beta_1<\frac{6}{K(b,A_*)}-1,\qquad\beta_2-\beta_1\stackrel{\beta_1
\to 0}{\longrightarrow}94.84,\quad\beta_2-\beta_1\stackrel{\beta_1\to\infty}
{\longrightarrow}5,\label{j40}
\eea
and the plot in Figure \ref{fig4}a has to be transformed by rescaling the 
coordinates $(b,{\cal B}_6)$ with inversion, $b\to \beta_1={\cal B}_6/b$, and 
shift, ${\cal B}_6\to \beta_2-\beta_1=2{\cal B}_6-1$ (see Figure \ref{fig4}b). 
Following sections \ref{s1} and \ref{s2}, the constraints (\ref{j38}) have to 
be supplemented by another double inequality $5<\beta_1<2(4d_1-1)$.
%%%%%%%%%%%%%%%%%%%%%%%%%%%%%%%%%%%%%%%%%%%
\begin{figure}[h!]\begin{center}\begin{tabular}{cc}
\psfig{figure=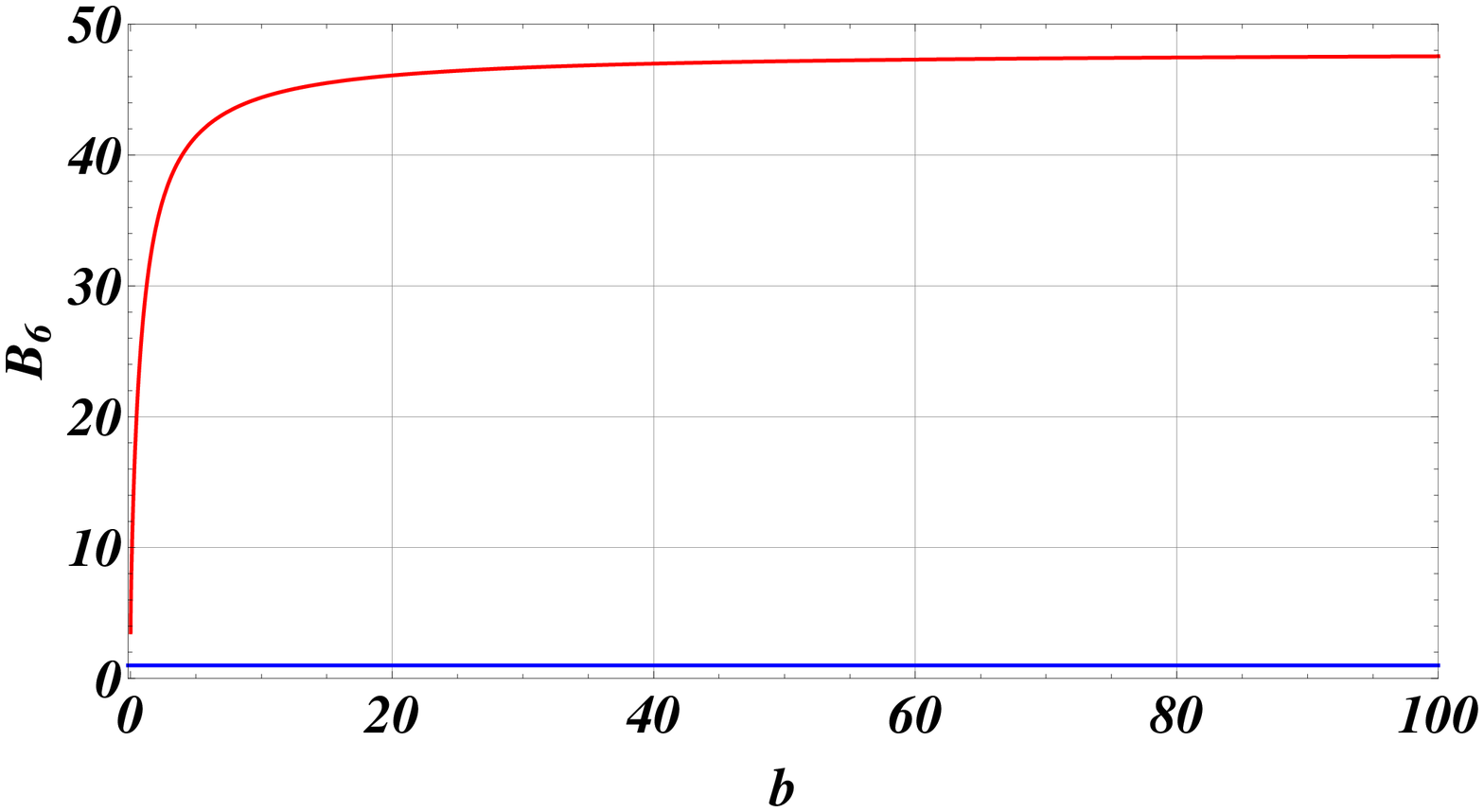,height=4.3cm}\hspace{-.5cm}&\hspace{-.5cm}
\psfig{figure=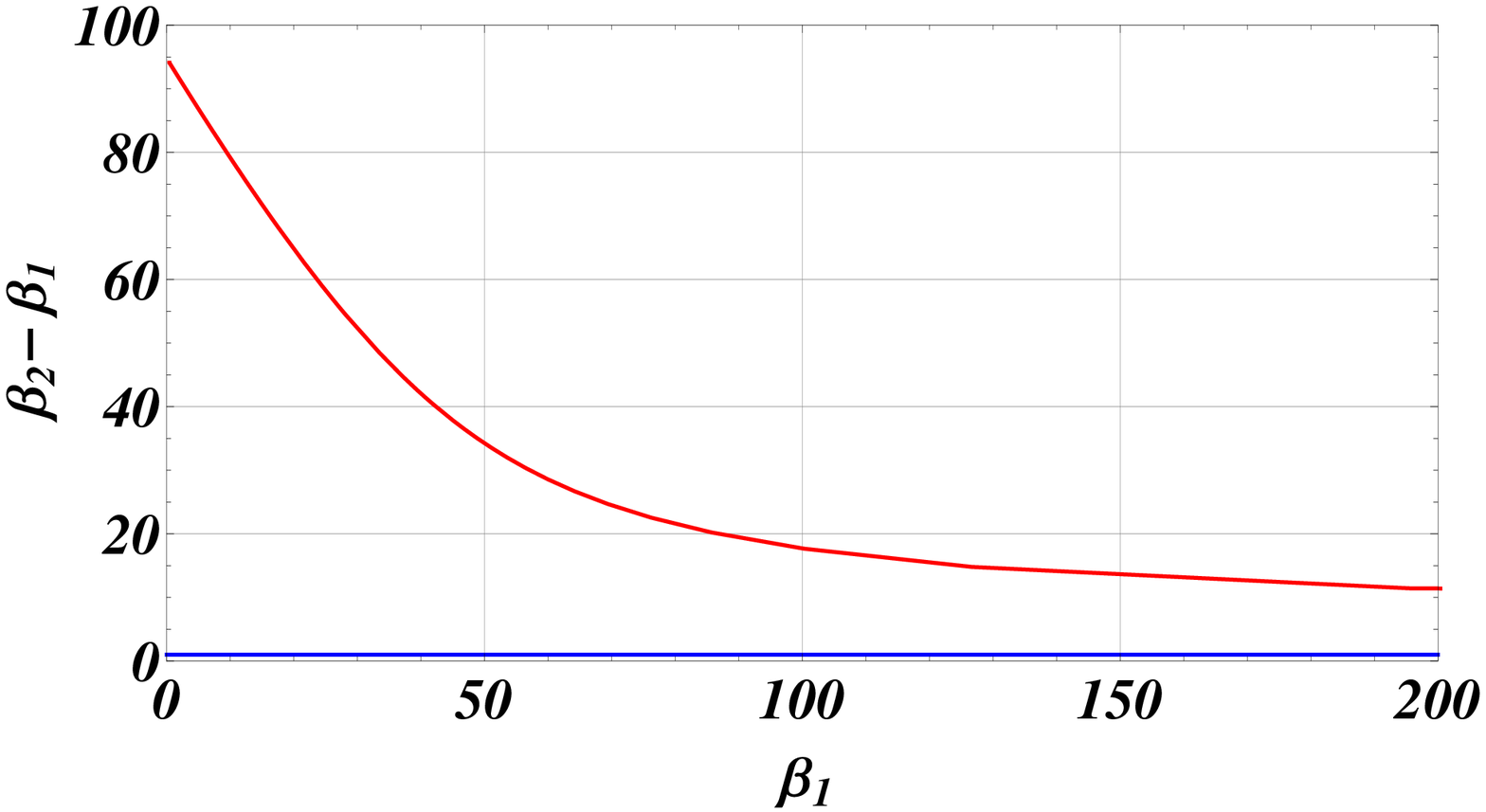,height=4.3cm}\\
(a) & (b)
\end{tabular}\end{center}
\vspace{-.5cm}
\caption{Plots of the lower ({\em blue}) and upper ({\em red}) bounds in the 
planes (a) $(b,{\cal B}_6)$ and (b) $(\beta_1,\beta_2-\beta_1)$.}\label{fig4}
\end{figure}
%%%%%%%%%%%%%%%%%%%%%%%%%%%%%%%%%%%%%%%%%%%%%%%%%

The double inequality (\ref{j40}) manifests a phenomenon, which does not exist
in symmetric (not CI) semigroups ${\sf S}_m$, generated by four \cite{fe016} 
and five \cite{fe018} integers, where inequalities $\widetilde{g_m}<g_m<
\overline{g_m}$, are always satisfied and independent of Betti's numbers 
($\beta_1=5$ for ${\sf S}_4$ and $\beta_1=\beta$ for ${\sf S}_5$):
\bea
\widetilde{g_m}<\lambda_m\sqrt[m-1]{\pi_m}<\overline{g_m},\quad\left\{
\begin{array}{l}\lambda_4=\sqrt[3]{25},\\\lambda_5=\sqrt[4]{192(\beta-1)/\beta},
\end{array}\right.\quad\left\{\begin{array}{l}\widetilde{g_m}=\sqrt[m-1]{(m-1)!
}\sqrt[m-1]{\pi_m},\\\overline{g_m}=(m-1)\sqrt[m-1]{\pi_m}.\end{array}\right.
\label{j41}
\eea
Note, that constraints (\ref{j40}) do not contradict Bresinsky's theorem 
\cite{bre75} on the arbitrary large finite value of $\beta_1$ for generic 
semigroup ${\sf S}_m$, $m\ge 4$. Below, we put forward some considerations 
about validity of (\ref{j40}) for Betti's numbers $\beta_1,\beta_2$ of 
symmetric (not CI) semigroup ${\sf S}_6$.

The double inequality (\ref{j40}) has arisen by comparison of $g_6$ with two 
other bounds $\widetilde{g_6}$ and $\overline{g_6}$ and, strictly speaking, a 
validity of (\ref{j40}) is dependent on how small is a discrepancy $\delta R(z)$
in Figure \ref{fig1}. If $\delta R(z)$ is not small enough and its neglecting 
in (\ref{j11}) is a far too rude approximation, then there may exist symmetric 
(not CI) semigroups ${\sf S}_6$ with Betti's numbers $\beta_1,\beta_2$, where 
(\ref{j40}) is broken. Such violation should indicate a necessity to improve 
the lower bound $g_6$ in (\ref{j36}) to restore the relationship $\widetilde{
g_6}<g_6<\overline{g_6}$. Note, that such improvement is very hard to provide 
even by replacing $A_*\to A$ in inequality (\ref{j12}), where $A_*<A<1$, and 
still preserving (\ref{j12}) with a new $A$. Such replacement leads again to 
(\ref{j36}) with $K(b,A)$ instead $K(b,A_*)$, i.e., the constraints on $\beta_1,
\beta_2$ still exist, even the area of admissible Betti's numbers becomes wider.

However, if there are no such symmetric (not CI) semigroups ${\sf S}_6$, where 
the double inequality (\ref{j40}) is broken, then there arises a much more deep 
question: why do the constraints on Betti's numbers exist. This problem is 
strongly related to the structure of minimal relations of the first and second 
syzygies in the minimal free resolution for the 1--dim Gorenstein (not CI) ring 
$k[{\sf S}_6]$ and has to be addressed in a separate paper.
%%%%%%%%%%%%%%%%%%%%%%%%%%%%%%%%%%%%%%%%%%%%%%%%%
\section{Symmetric (not CI) semigroups ${\sf S}_6$ with the $W$ and $W^2$ 
properties}\label{s5}
%%%%%%%%%%%%%%%%%%%%%%%%%%%%%%%%%%%%%%%%%%%%%%%%%
In \cite{fe018}, we introduced a notion of the {\em W} property for the 
$m$-generated symmetric (not CI) semigroups ${\sf S}_m$ satisfying Watanabe's 
Lemma \cite{wata73}. We recall this Lemma together with the definition of 
the {\em W} property and two other statements relevant in this section.
%%%%%%%%%%%%%%%%%%%%%%%%%%%%%%%%%%%%%%%%%%%%%%%%%%%%%%%%
\begin{lemma}\label{le2}{\rm (\cite{wata73}).}
%%%%%%%%%%%%%%%%%%%%%%%%%%%%%%%%%%%%%%%%%%%%%%%%%%%%%%%%
Let a semigroup $S_{m-1}\!=\!\la\delta_1,\ldots,\delta_{m-1}\ra$ be given and 
$a\in{\mathbb Z}$, $a>1$, such that $\gcd(a,d_m)=1$, $d_m\in S_{m-1}\!\setminus
\!\{\delta_1,\ldots,\delta_{m-1}\}$. Consider a semigroup $S_m\!=\!\la a\delta
_1,\ldots,a\delta_{m-1},d_m\ra$ and denote it by $S_m\!=\!\la aS_{m-1},d_m\ra$. 
Then $S_m$ is symmetric if and only if $S_{m-1}$ is symmetric, and $S_m$ is 
symmetric CI if and only if $S_{m-1}$ is symmetric CI.
\end{lemma}
%%%%%%%%%%%%%%%%%%%%%%%%%%%%%%%%%%%%%%%%%%%%%%%%%%%%%%%%
\begin{corollary}\label{co1}{\rm (\cite{fe018}).}
%%%%%%%%%%%%%%%%%%%%%%%%%%%%%%%%%%%%%%%%%%%%%%
Let a semigroup $S_{m-1}\!=\!\la\delta_1,\ldots,\delta_{m-1}\ra$ be given and 
$a\in{\mathbb Z}$, $a>1$, such that $\gcd(a,d_m)=1$, $d_m\in S_{m-1}\!\setminus
\!\{\delta_1,\ldots,\delta_{m-1}\}$. Consider a semigroup $S_m\!=\!\!\la aS_
{m-1},d_m\ra$. Then $S_m$ is symmetric (not CI) if and only if $S_{m-1}$ is 
symmetric (not CI).
\end{corollary}
%%%%%%%%%%%%%%%%%%%%%%%%%%%%%%%%%%%%%%%%%%%%%%%%%%%%%%%%
\begin{definition}\label{de1}{\rm (\cite{fe018}).}
{\rm A symmetric (not CI) semigroup $S_m$ has {\em the property W} if there
exists another symmetric (not CI) semigroup $S_{m-1}$ giving rise to $S_m$     
by the construction, described in Corollary \ref{co1}.}
\end{definition}
%%%%%%%%%%%%%%%%%%%%%%%%%%%%%%%%%%%%%%%%%%%%%%%%%%%%%%%%
\begin{theorem}\label{th2}{\rm (\cite{fe018}).}
%%%%%%%%%%%%%%%%%%%%%%%%%%%%%%%%%%%%%%%%%%%%%%%%%%%%%%%%
A minimal {\em edim} of symmetric (not CI) semigroup $S_m$ with the property
{\em W} is $m=5$.
\end{theorem}
%%%%%%%%%%%%%%%%%%%%%%%%%%%%%%%%%%%%%%%%%%%%%%%%%%%%%%%%
In this section we study the symmetric (not CI) semigroups $S_6$ satisfying 
Watanabe's Lemma \cite{wata73}. To distinguish such semigroups from the 
rest of symmetric (not CI) semigroups $S_6$ without the property {\em W} we 
denote them by ${\sf W}_6$.
%%%%%%%%%%%%%%%%%%%%%%%%%%%%%%%%%%%%%%%%%%%%%%%%%%%%%%%%
\begin{lemma}\label{le3}
Let two symmetric (not CI) semigroups ${\sf W}_6\!=\!\la a{\sf S}_5,d_6\ra$ and 
${\sf S}_5\!=\!\la q_1,\ldots,q_5\ra$ be given and $\gcd(a,d_6)\!=\!1$, $d_6\in
{\sf S}_5\!\setminus\!\{q_1,\ldots,q_5\}$. Let the lower bound $F_{6w}$ of the 
Frobenius number $F({\sf W}_6)$ of the semigroup ${\sf W}_6$ be represented as, 
$F_{6w}\!=\!g_{6w}-\left(a\sum_{j=1}^5q_j+d_6\right)$. Then
\vspace{-.2cm}
\bea
g_{6w}=a\left(\lambda_5\sqrt[4]{\pi_5(q)}+d_6\right),\qquad\pi_5(q)=\prod_{j=1}
^5q_j.\label{j42}
\eea
where $\lambda_5$ is defined in (\ref{j41}).
\end{lemma}
%%%%%%%%%%%%%%%%%%%%%%%%%%%%%%%%%%%%%%%%%%%%%%%%%%%%%%%%
\begin{proof}
Consider a symmetric (not CI) numerical semigroup ${\sf S}_5$ generated by five
integers (without the $W$ property), and apply the recent result \cite{fe018} 
on the lower bound $F_5$ of its Frobenius number, $F({\sf S}_5)$,
\bea
F({\sf S}_5)\ge F_5,\qquad F_5=h_5-\sum_{j=1}^5q_j,\qquad h_5=\lambda_5
\sqrt[4]{\pi_5(q)}.\label{j43}
\eea
The following relationship between the Frobenius numbers $F({\sf W}_6)$ and
$F({\sf S}_5)$ was derived in \cite{bra62}:
\bea
F({\sf W}_6)=aF({\sf S}_5)+(a-1)d_6.\label{j44}
\eea
Substituting $F({\sf W}_6)\!=\!g\!-\!\left(a\sum_{j=1}^5q_j+d_6\right)$ and the 
representation (\ref{j43}) for $F({\sf S}_5)$ into (\ref{j44}), we obtain
\bea
g-a\sum_{j=1}^5q_j-d_6=ah_5-a\sum_{j=1}^5q_j+(a-1)d_6\qquad\rightarrow\qquad 
g=a(h_5+d_6).\label{j45}
\eea
Comparing the last equality in (\ref{j45}) with the lower bound of $h_5$ in
(\ref{j43}), we arrive at (\ref{j42}).
\end{proof}
%%%%%%%%%%%%%%%%%%%%%%%%%%%%%%%%%%%%%%%%%%%%%%%%%%%%%%%%
Following Corollary \ref{co1}, let us apply the construction of a symmetric 
(not CI) semigroup $S_m$ with the $W$ property to a symmetric (not CI) semigroup
$S_{m-1}$, which already has such property.
%%%%%%%%%%%%%%%%%%%%%%%%%%%%%%%%%%%%%%%%%%%%%%%%%%%%%%%%
\begin{definition}\label{de2}
{\rm A symmetric (not CI) semigroup $S_m$ has {\em the property $W^2$} if there
exist two symmetric (not CI) semigroup $S_{m-1}=\la q_1,\ldots,q_{m-1}\ra$ and 
$S_{m-2}=\la p_1,\ldots,p_{m-2}\ra$ giving rise to $S_m$ by the construction, 
described in Corollary \ref{co1},
\bea
&&S_m\!=\!\la a_1S_{m-1},d_m\ra,\hspace{1.1cm}d_m\in S_{m-1}\!\setminus\!\{q_1,
\ldots,q_{m-1}\},\qquad\gcd(a_1,d_m)\!=\!1,\nonumber\\
&&S_{m-1}\!=\!\la a_2S_{m-2},q_{m-1}\ra,\quad q_{m-1}\in S_{m-2}\!\setminus\!
\{p_1,\ldots,p_{m-2}\},\quad\gcd(a_2,q_{m-1})\!=\!1.\nonumber
\eea}
\end{definition}
%%%%%%%%%%%%%%%%%%%%%%%%%%%%%%%%%%%%%%%%%%%%%%%%%%%%%%%%
\begin{theorem}\label{th3}
%%%%%%%%%%%%%%%%%%%%%%%%%%%%%%%%%%%%%%%%%%%%%%%%%%%%%%%%
A minimal {\em edim} of symmetric (not CI) semigroup $S_m$ with the property
{\em $W^2$} is $m=6$.
\end{theorem}
%%%%%%%%%%%%%%%%%%%%%%%%%%%%%%%%%%%%%%%%%%%%%%%%%%%%%%%%
\begin{proof}
This statement follows if we combine Definition \ref{de2} and Theorem \ref{th2}.
\end{proof}
%%%%%%%%%%%%%%%%%%%%%%%%%%%%%%%%%%%%%%%%%%%%%%%%%%%%%%%%
In this section we denote the symmetric (not CI) semigroups $S_6$ with the 
property {\em $W^2$} by ${\sf W}^2_6$.
%%%%%%%%%%%%%%%%%%%%%%%%%%%%%%%%%%%%%%%%%%%%%%%%%%%%%%%%
\begin{lemma}\label{le4}
Let three symmetric (not CI) semigroups ${\sf W}^2_6=\la a_1{\sf W}_5,d_6\ra$, 
${\sf W}_5=\la a_2{\sf S}_4,q_5\ra$, and ${\sf S}_4=\la p_1,\ldots,p_4\ra$, 
where $q_j=a_2p_j$, $1\le j\le 4$, be given in such a way that
$$ 
d_6\in{\sf W}_5\!\setminus\!\{q_1,\ldots,q_5\},\quad q_5\in{\sf S}_4\!\setminus
\!\{p_1,\ldots,p_4\},\quad\gcd(a_1,d_6)=\gcd(a_2,q_5)=1.
$$ 
Let the lower bound $F_{6w}$ of the Frobenius number $F({\sf W}^2_6)$ of the 
semigroup ${\sf W}^2_6$ be represented as, $F_{6w^2}\!=\!g_{6w^2}-\left(a_1a_2
\sum_{j=1}^4p_j+a_1q_5+d_6\right)$. Then
\bea
g_{6w^2}=a_1\left[a_2\left(\lambda_4\sqrt[3]{\pi_4(p)}+q_5\right)+d_6\right],
\qquad\pi_4(p)=\prod_{j=1}^4p_j,\label{j46}
\eea
where $\lambda_4$ is defined in (\ref{j41}).
\end{lemma}
%%%%%%%%%%%%%%%%%%%%%%%%%%%%%%%%%%%%%%%%%%%%%%%%%%%%%%%%
\begin{proof}
By Lemma 2 in \cite{fe018}, the lower bound $F_{5w}$ of its Frobenius number
$F({\sf W}_5)$ of the symmetric (not CI) semigroup ${\sf W}_5$ reads:
\bea
F_{5w}=g_{5w}-\left(a_2\sum_{j=1}^4p_j+q_5\right),\qquad
g_{5w}=a_2\left(\lambda_4\sqrt[3]{\pi_4(p)}+q_5\right).\label{j47}
\eea
Consider a symmetric (not CI) semigroup ${\sf W}^2_6$, generated by six 
integers, and make use of a relationship between the Frobenius numbers 
$F({\sf W}^2_6)$ and $F({\sf W}_5)$ derived in \cite{bra62}:
\bea
F({\sf W}^2_6)=a_1F({\sf W}_5)+(a_1-1)d_6.\label{j48}
\eea
Substituting $F({\sf W}^2_6)\!=\!g_{6w^2}-\left(a_1a_2\sum_{j=1}^4p_j+
a_1q_5+d_6\right)$ and the representation (\ref{j47}) for $F({\sf S}_5)$ 
into (\ref{j48}), we obtain
\bea
g_{6w^2}-a_1a_2\sum_{j=1}^4p_j-a_1q_5-d_6=a_1\left[g_{5w}-
\left(a_2\sum_{j=1}^4p_j+q_5\right)\right]+(a_1-1)d_6.\label{j49}
\eea
Simplifying the last equality (\ref{j49}), we arrive at (\ref{j46}).
\end{proof}
Among the subsets $\{{\sf W}^2_6\}$, $\{{\sf W}_6\}$ and the entire set 
$\{{\sf S}_6\}$ of symmetric (not CI) semigroups, generated by six integers, 
the following containment holds:
$$
\{{\sf W}^2_6\}\subset\{{\sf W}_6\}\subset\{{\sf S}_6\}.
$$
%%%%%%%%%%%%%%%%%%%%%%%%%%%%%%%%%%%%%%%%%%%%%%%%%%%%%%%%
Below we present twelve symmetric (not CI) semigroup generated by six integers:
$V_1,V_2,V_3,V_4$ -- without the $W$ property, $V_5,V_6,V_7,V_8$ -- with the 
$W$ property, $V_9,V_{10},V_{11},V_{12}$ -- with the $W^2$ property.
\bea
&&V_1=\la 7,9,11,12,13,15\ra,\qquad V_5=\la 12,20,28,30,38,41\ra,\quad
V_9=\la 30,33,36,37,42,48\ra,\nonumber\\
&&V_2=\la 7,9,10,11,12,13\ra,\qquad V_6=\la 12,20,28,38,46,47\ra,\quad
V_{10}\!=\!\la 42,45,48,54,59,78\ra,\nonumber\\
&&V_3=\la 12,13,14,15,17,19\ra,\quad V_7=\la 14,24,26,36,46,49\ra,\quad 
V_{11}\!=\!\la 40,42,48,54,71,78\ra,\nonumber\\
&&V_4=\la 12,13,14,15,18,19\ra,\quad V_8=\la 38,46,58,62,74,79\ra,\quad
V_{12}\!=\!\la 46,48,75,78,90,102\ra.\nonumber
\eea
We give a comparative Table 1 for the largest degree $g$ of syzygies and its 
lower bounds $g_6$, $g_{6w}$, $g_{6w^2}$ and $\widetilde{g_6}$, calculated by 
formula (\ref{j37}).
\begin{center}
\begin{tabular}
{|c||c|c|c|c|c|c|c|c|c|c|c|c|}\hline
\multicolumn{1}{|c||}{}&\multicolumn{4}{|c|}{---}&\multicolumn{4}{|c|}{
{\em $W$ property}}&\multicolumn{4}{|c|}{{\em $W^2$ property}}\\\hline
${\sf S}_6$&$V_1$&$V_2$&$V_3$&$V_4$&$V_5$&$V_6$&$V_7$&$V_8$&$V_9$&$V_{10}$&
$V_{11}$&$V_{12}$\\\hline\hline
$\beta_1$ & 13 & 14 & 10 & 10 & 8 & 9 & 10 & 14 & 7 & 7 & 7 & 7\\ \hline
$\beta_2$ & 31 & 35 & 19 & 22 & 19 & 18  & 23 & 37 &16 &16 &16 &16 \\ \hline
${\cal B}_6$&9.5 & 11 & 5 & 6.5 & 6 & 5 & 7 & 12 &5 &5 &5 & 5\\ \hline\hline
$g$ & 84 & 77 & 125 & 126 & 256 & 292 & 302 & 638 &387 &603 &598 & 816\\\hline
$g_{6w^2}$ &--&--&--&--&--&--&--&--&385.6 &595.3 & 590.3& 811.2\\ \hline
$g_{6w}$ &--&--&--&--&240.4&271.2&286&609.2&359.8&554.8&548&746.9\\\hline
$g_6$&55&49.6&88&86.5&173.3&196&196.6&395.4&274.4&420.8&426.6&586.6\\\hline
$\widetilde{g_6}$ & 45.5 & 42 & 66.2 & 66.9 &130.4 &146 &153.4 &338 &199.9 
&306.5 & 310.7& 427.2\\\hline
\end{tabular}
\end{center}

\begin{center}
Table 1. The largest degree $g$ of syzygies for symmetric (not CI) semigroups 
${\sf S}_6$ with different Betti's numbers $\beta_1,\beta_2$ and its lower 
bounds $g_6$, $g_{6w}$, $g_{6w^2}$, $\widetilde{g_6}$.
\end{center}
For symmetric (not CI) semigroups ${\sf W}^2_6$, presented in Table 1, there 
following inequalities hold:
\bea
g>g_{6w^2}>g_{6w}>g_6>\widetilde{g_6}.\label{j50}
\eea
For the rest of symmetric (not CI) semigroups ${\sf W}_6$ and ${\sf S}_6$ the 
bounds $g_{6w^2}$ and $g_{6w}$ are skipped in inequalities (\ref{j50}) depending
on the existence (or absence) of the {\em W} property in these semigroups. It 
is easy to verify that the Betti numbers of all semigroups from Table 1 
satisfy the constraints (\ref{j40}).
%%%%%%%%%%%%%%%%%%%%%%%%%%%%%%%%%%%%%%%%%%%%%%%%%%%%%%%%%%%
\section*{Acknowledgement}
%%%%%%%%%%%%%%%%%%%%%%%%%%%%%%%%%%%%%%%%%%%%%%%%%%%%%%%%%%%
The research was partly supported by the Kamea Fellowship.
%%%%%%%%%%%%%%%%%%%%%%%%%%%%%%%%%%%%%%%%%%%%%%%%%%%%%%%%%%%

\end{document}